\documentclass[a4paper,11pt]{amsart}

\usepackage{amsmath,amssymb}
\usepackage{amscd}
\usepackage[all,cmtip]{xy}
\usepackage{framed}
\usepackage{hyperref}
\usepackage[inline]{enumitem}  
\usepackage{float}
\usepackage{multicol}
\usepackage{array}
\usepackage{makecell}
\newcolumntype{P}[1]{>{\centering\arraybackslash}p{#1}}
\newcolumntype{M}[1]{>{\centering\arraybackslash}m{#1}}
\usepackage{lscape}
\usepackage{times}

\usepackage{setspace}

\theoremstyle{plain}
\newtheorem{thm}{Theorem}[section]
\newtheorem{theorem}[thm]{Theorem}

\newtheorem{lemma}[thm]{Lemma}

\theoremstyle{definition}

\newtheorem{remark}[thm]{Remark}

\newtheorem{definition}[thm]{Definition}
\newtheorem*{claim*}{Claim}
\newtheorem*{claimproof*}{Proof of Claim}

\newtheorem{example}[thm]{Example}

\numberwithin{equation}{section}
\newcommand{\sL}{{\mathcal L}}
\newcommand{\sO}{{\mathcal O}}
\newcommand{\C}{{\mathbb C}}
\newcommand{\BG}{{\mathbb G}}
\newcommand{\BH}{{\mathbb H}}
\newcommand{\BP}{{\mathbb P}}

\newcommand{\R}{{\mathbb R}}
\newcommand{\Z}{{\mathbb Z}}

\newcommand{\SL}{{\rm SL}}
\newcommand{\PSL}{{\rm PSL}}

\newcommand{\Amp}{\operatorname{Amp}}
\newcommand{\Aut}{\operatorname{Aut}}

\newcommand{\Grass}{\operatorname{Grass}}
\newcommand{\rank}{\operatorname{rank}}

\newcommand{\id}{\operatorname{id}}

\newcommand{\trace}{\operatorname{trace}}

\def\NS{\mathop{\rm NS}\nolimits}

\begin{document}

\title[Automorphisms of Hilbert schemes]{Automorphisms of Hilbert schemes of Cayley's K3 surfaces}
\author{Kwangwoo Lee}
\address{Department of Mathematics, Ewha Womans University, 52 Ewhayeodae-gil, Seodaemun-gu, Seoul, Republic of Korea}
\email{klee0222@gmail.com}

\thanks{The author was supported by National Research Foundation of Korea(NRF-2020R1I1A1A01073662) and National Research Foundation of Korea(NRF-2021R1A4A3033098) grant funded by the Korea government(MSIT)}

\date{\today}

\maketitle

\begin{abstract}
We prove that the automorphism group of Hilbert square of a Cayley's K3 surface of Picard number $2$ is the free product of three cyclic groups of order two. The generators are three Beauville involutions.
\end{abstract}

\section{Introduction}
Let $X$ be a nonsingular projective surface over $\C$, and let $X^{[m]}$ be the Hilbert scheme of $m$ points on $X$, the space parametrizing $0$-dimensional subschemes $Z\subset X$ of length $m$. By the result of Beauville, the Hilbert scheme of points of a K3 surface is a hyperk$\ddot{\text{a}}$hler manifold (\cite{B2}).  A \textit{hyperk$\ddot{\text{a}}$hler manifold} is a simply connected compact K$\ddot{\text{a}}$hler manifold $Y$ such that $H^0(Y,\Omega_Y^2)=\mathbb{C}\sigma_Y$, where $\sigma_Y$ is an everywhere non-degenerate holomorphic $2$-form on $Y$. In particular, a K3 surface is a $2$-dimensional hyperk$\ddot{\text{a}}$hler manifold. For a K3 surface $X$, the Hilbert scheme $X^{[m]}$ of length $m$ of $X$ is a hyperk$\ddot{\text{a}}$hler manifold of dimension $2m$ with \textit{Beauville-Bogomolov-Fujiki's form} (\cite{B2}) on $H^2(X^{[m]},\mathbb{Z})$ and we have a natural Hodge isometry
\begin{equation} \label{BBF}
H^2(X^{[m]},\mathbb{Z})\cong H^2(X,\mathbb{Z})\oplus \mathbb{Z}[e],
\end{equation}
where $e:=E/2$ with self-intersection number $(e^2)=-2m+2$ and $E$ is the exceptional divisor of the \textit{Hilbert-Chow} morphism 
\begin{equation}\label{Hilb-Chow}
\epsilon: X^{[m]}\rightarrow X^{(m)}=X^m/S_m.
\end{equation}
Here $X^{(m)}$ is the $m$-th symmetric product of $X$ defined by the quotient of Cartesian product $X^m$ by the symmetric group $S_m$.
Moreover, we have the natural injections of automorphisms
\begin{equation} \label{inclusion}
\Aut(X)\hookrightarrow \Aut(X^{[m]})
\end{equation}
and  
\begin{equation}\label{injection}
\Aut(X^{[m]})\hookrightarrow O(H^2(X^{[m]},\mathbb{Z}),q_{X^{[m]}}), \hspace{5mm} f\mapsto f^*,
\end{equation}
where $O(H^2(X^{[m]},\mathbb{Z}),q_{X^{[m]}})$ is the isometry group of $H^2(X^{[m]},\mathbb{Z})$ with the quadratic form $q_{X^{[m]}}$.

In \cite{O1}, Oguiso considered a K3 surface $S$ with Picard lattice $NS(S)=\Z h_1\oplus \Z h_2$ where the intersection matrix 
\begin{equation}\label{Picard lattice of Oguiso}
[(h_i,h_j)]=\begin{bmatrix} 4& 2\\ 2&-4\end{bmatrix}.
\end{equation} 
He constructed an isometry of $NS(S)$ given by the multiplication of $\eta^6$ on $NS(S)$, where $h_2=\eta:=\frac{1+\sqrt{5}}{2}$ is a basis of the lattice $NS(S)$ together with $h_1:=1$. Then by the global Torelli theorem for K3 surfaces, it induces an automorphism $g$ of $S$ with positive entropy without fixed points. It is also known as the \textit{Cayley-Oguiso} automorphism in \cite{FGvGvL}. This was introduced by Cayley. In \cite{Cay}, he constructed a K3 surface with an automorphism of infinite order. Let $X$ be a Cayley's K3 surface: the determinantal quartic surface in $\BP^3$ with homogeneous coordinates $\textbf{x}=[x_0:x_1:x_2:x_3]$ defined by
\begin{equation}\label{Cayley's K3 surface}
X:=(\det M(\textbf{x})=0)\subset \BP^3_{\textbf{x}},
\end{equation}
where $M(\textbf{x})$ is the generic $4\times 4$ matrix whose $(k,j)$ entry is $\Sigma_i a_{ijk}x_i$.
By a genericity assumption of $M$, $\rank M(\textbf{x})=3$ for all $\textbf{x}\in X$ and $X$ is a smooth quartic K3 surface. In particular, let $X$ be a Cayley's K3 surface with the rank of $NS(X)=2$. Then by \cite[Proposition 2.2]{FGvGvL}, $X\cong S$. 

For a very ample divisor $D_n:=\eta^{2n}$ with $n\in \Z$ on $X$, we have an isomorphism $\phi_{D_n}:X\xrightarrow{\cong} X_n\subset \BP^3$ (cf. \cite[\S 5.1]{L}). Then $\phi_{D_n}$ defines a Beauville birational involution $\iota_n$ on the Hilbert scheme $X^{[2]}$ of $X$. Since $X$ has no line, $\iota_n$ is a biregular involution of $X^{[2]}$ by \cite{B1}. In \cite[Theorem 2]{L}, the author showed that all Beauville involutions are generated by $\iota_0,\iota_1,\iota_2$ and the \textit{natural} automorphism $g$ induced from $X$, i.e., an image of \eqref{inclusion}. 

In this paper, we will show the follow theorem:
\begin{theorem}\label{main theorem1}
$\Aut(X^{[2]})\cong \langle \iota_0,\iota_1,\iota_2\rangle\cong \Z_2\ast \Z_2\ast \Z_2$. 
\end{theorem}

The first equivalence in Theorem \ref{main theorem1} is also known in \cite{O2} without proof. Our strategies are the followings. For the first equivalence, we use the characterization of natural automorphisms. Next we use hyperbolic geometry to prove the second equivalence.

\begin{remark}\label{generic Picard number case}
For Hilbert squares of a K3 surface with Picard number one, the automorphism group is either trivial or $\Z_2$ generated by a non-symplectic involution \cite{BCNS,Cat}.
\end{remark}

\section{Preliminaries}
In this section we recall some useful results on Hilbert squares of projective K3 surfaces and on hyperbolic geometry which we will use in the following sections.
\subsection{Hilbert squares of K3 surfaces}
For a K3 surface $X$, the Hilbert scheme $X^{[m]}$ of length $m$ of $X$ is a hyperk$\ddot{\text{a}}$hler manifold of dimension $2m$ by Beauville \cite{B2}. There exists an injective morphism $i: H^2(X,\C)\rightarrow H^2(X^{[m]},\C)$ such that 
\begin{equation}
H^2(X^{[m]},\C)=i(H^2(X,\C))\oplus \C[E]
\end{equation}
 with a quadratic form $q_{X^{[m]}}$ on $X^{[m]}$ satisfying $q_{X^{[m]}}(i(\alpha))=q(\alpha)$ for $\alpha\in H^2(X,\C)$, where $q$ is the quadratic form on $X$. Note that $E$ is the exceptional divisor of \eqref{Hilb-Chow}. For $e:=E/2$, $H^2(X^{[m]},\mathbb{Z})=i(H^2(X,\Z))\oplus \Z[e]$.

For any line bundle $L$ on $X$, the tensor product $\otimes_{i=1}^m p_i^*(L)$ has a natural $S_m$-invariant structure, and taking $S_m$-invariants defines a line bundle $L(m)$ on $X^{(m)}$, where $p_i:X^m\rightarrow X$ is the projection onto the $i$-th factor. We also define the pullback via the Hilbert-Chow morphism \eqref{Hilb-Chow} to the Hilbert scheme:
\begin{equation}
L[m]:=\epsilon^*L(m).
\end{equation}
By construction we have the first Chern class $c_1(L[m])=i(c_1(L))$.

For any $0$-dimensional subscheme $(Z,\mathcal{O}_Z)$ of length $h^0(\mathcal{O}_Z)=k+1$, if the restriction map $H^0(L)\rightarrow H^0(L\otimes \mathcal{O}_Z)$ is surjective, then $L$ is called \textit{k-very ample} for an integer $k\geq 0$ (\cite{BS}). By definition, $L$ is $0$-very ample if and only if $L$ is generated by its global sections, and $L$ is $1$-very ample if and only if $L$ is very ample.
For the Grassmannian $\Grass(m,H^0(L))$ of all $m$-dimensional quotients of $H^0(L)$, we have a rational map
\begin{equation}\label{map to Grass}
\phi_L: X^{[m]}\dashrightarrow \Grass(m,H^0(L)),
\end{equation}
sending $(Z,\sO_Z)\in X^{[m]}$ to the quotient $H^0(L)\rightarrow H^0(L\otimes \sO_Z)$.
If $L$ is $(m-1)$-very ample, then $\phi_L$ is a morphism and $\phi_L$ is embedding if and only if $L$ is $m$-very ample (cf. \cite{CG}).



Let $X$ be a K3 surface with Picard lattice $NS(X)=\Z h_1\oplus \Z h_2$ whose intersection matrix is given in \eqref{Picard lattice of Oguiso}.
Since $X$ has no $(-2)$-curve, the ample cone of $X$ is just the positive cone, i.e.,
\begin{equation}\label{ample cone of X}
\begin{array}{lcl}
\Amp(X)&=&\{(x,y)\in NS(X)\otimes \R\mid (x,y)^2>0 \text{ and } x>0\}\\
&=&\{(x,y)\in NS(X)\otimes \R\mid x>0 \text{ and }\frac{1-\sqrt{5}}{2}<\frac{y}{x}<\frac{1+\sqrt{5}}{2}\}.\\
\end{array}
\end{equation}

For an ample line bundle $L$ on $X$, $\otimes_{i=1}^m p_i^*(L)$ is ample on the Cartesian product $X^m$, hence $L(m)$ is ample on the symmetric product $X^{(m)}$. Since the Hilbert-Chow morphism is birational, the induced line bundle $L[m]$ is big and nef on $X^{[m]}$. However, since $L[m]$ has degree zero on the fibers of the Hilbert-Chow morphism, $L[m]$ is not ample. Hence, the line bundle $L[m]$ lies on the boundary of the nef cone of $X^{[m]}$. 

\begin{lemma}\label{ample divisors}
For any ample line bundle $L$ on the K3 surface $X$ with Picard lattice $NS(X)=\Z h_1\oplus \Z h_2$, the line bundle $L[2]-e$ is ample on $X^{[2]}$.
\end{lemma}
\begin{proof}
Let $L$ be an ample line bundle on $X$.
 In \cite[Theorem 1.1]{Kn}, Knutsen showed that $L$ is $k$-very ample if and only if $L^2\geq 4k$ and there exists no effective divisor $D$ satisfying the conditions $(*)$ below:
\begin{equation}
\begin{array}{lcl}
\tag{*}& 2D^2\overset{(i)}{\leq} L.D\leq D^2+k+1\overset{(ii)}{\leq} 2k+2\\
&\text{ with equality in } (i) \text{ if and only if } L\sim 2D \text{ and } L^2\leq 4k+4,\\
&\text{ and equality in } (ii) \text{ if and only if } L\sim 2D \text{ and } L^2=4k+4.\\
\end{array}
\end{equation}
Since any effective divisor $D$ on $X$ has $D^2\geq 4$, for $k=2$, $L^2\geq 8$ implies that $L$ is $2$-very ample.

If $L$ is $2$-very ample line bundle, then $\phi_L$ in \eqref{map to Grass} is an embedding, hence $\phi_L^*\sO_{\BG}(1)$ is ample on $X^{[2]}$, where $\BG:=\Grass(2,H^0(L))$. As explained in \cite[\S 2]{BC}, it is $L[2]-e$. Next since $L^2>0$ and $L^2\equiv 0 \mod 4$, we suppose that $L^2=4$. In \cite[Lemma 12]{L}, we showed that any such $L=D_n:=\eta^{2n}$ ($n\in \Z$) and it is very ample, hence it defines an embedding $X\hookrightarrow \BP^3$. By this embedding we identify $\phi_{L}: X^{[2]}\rightarrow \Grass(2,H^0(L))$ with 
\begin{equation}
\phi_{L}: X^{[2]}\rightarrow \Grass(1,\BP^3), Z\mapsto \langle Z\rangle
\end{equation}
where $\langle Z\rangle$ is the one dimensional span of $Z$ in $\BP^3$.
Moreover, since $X$ has no line, $\phi_L$ is a finite morphism onto its image. Therefore, $\phi^*\sO_{\BG}(1)=L[2]-e$ is ample on $X^{[2]}$.
\end{proof}


\begin{theorem}\cite[Theorem 1]{HT}\label{ample divisors on Hilbert squares}
Let $F$ be a projective algebraic variety deformation equivalent to Hilbert schemes of a K3 surface. Fix an ample divisor $H$ on $F$. A divisor $C$ on $F$ is ample if $C.D>0$ for each divisor class $D$ satisfying $H.D>0$ and either
\begin{enumerate}
\item $D^2\geq -2$; or
\item $D^2=-10$ and $(D,H^2(F,\Z))=2\Z$.
\end{enumerate}
\end{theorem}

\begin{lemma}\label{no -10 class}
For $X$ in Lemma \ref{ample divisors}, there is no divisor $D$ of square $-10$ on $X^{[2]}$.
\end{lemma}
\begin{proof}
For $D=(x,z,y)$, suppose that $D^2=4(x^2+xy-y^2)-2z^2=-10$. Then since $z$ is odd, let $z=2z'+1$ and then we have $x^2+xy-y^2=2z'(z'+1)-2$, i.e.,
$x^2+xy-y^2\equiv 2 \mod 4.$
Now if $x$ ($y$, resp.) is even, then $y$ ($x$, resp.) is also even, hence a contradiction. So both $x,y$ are odd, but then $x^2+xy-y^2$ is odd, a contradiction.
\end{proof} 

In \cite[Theorem 6.6]{O2}, Oguiso showed that $g=\iota_0\circ \iota_1\circ \iota_2$ without proof. Here we used the same notation $g$ for the induced automorphism on $X^{[2]}$ from $X$ via \eqref{inclusion}. In \cite[Remark 2]{L}, we confirmed this relation.
In \cite[Theorem 2]{L}, the author also found the relations between Beauville involutions and the natural automorphism $g$ as follows. 
\begin{theorem}\cite[Theorem 2]{L}\label{relations} 
For $k=0,1,$ or $2$ and $l\in \Z$,
\begin{equation}
\iota_{3l+k}=g^l\circ \iota_k \circ g^{-l}.
\end{equation}
\end{theorem}


Boissi$\grave{\text{e}}$re and Sarti proved that if $X$ is a K3 surface, then an automorphism $f\in \Aut(X^{[n]})$ is natural if and only if it preserves the diagonal (cf. \cite{BS}). 



\subsection{Hyperbolic geometry}
Recall that the $n$-dimensional hyperbolic space, usually denoted by $\BH^{n}$, is the unique simply connected, $n$-dimensional complete Riemannian manifold with a constant negative sectional curvature equal to $-1$. One of the models of hyperbolic plane is the upper half-plane
\begin{equation}
\BH^2=\{(x,y)\in \R^2\mid y>0\}.
\end{equation}
Recall that the group of all orientation-preserving isometries of $\BH^2$ is
\begin{equation}\label{Mobius transformations}
\text{M}\ddot{\text{o}}\text{b}^+(\BH^2):=\{g(z)=\frac{az+b}{cz+d}\mid a,b,c,d\in \R, ad-bc=1\},
\end{equation}
where $z=x+\sqrt{-1}y$, a complex number. Note that such an isometry corresponds to a matrix 
\begin{equation}
\begin{bmatrix} a& b\\c&d \end{bmatrix}\in \SL(2,\R).
\end{equation}
Moreover, since two matrices in $\SL(2,\R)$ induce the same isometry if and only if their matrices differ by a factor of $\pm 1$, hence $\text{M}\ddot{\text{o}}\text{b}^+(\BH^2)\cong \PSL(2,\R)$.

Any orientation-preserving isometry of $\BH^2$ is one of the following three types. 
\begin{definition}
Let $g(\neq \id)$ be any M$\ddot{\text{o}}$bius transformation in \eqref{Mobius transformations}. Then
\begin{enumerate}
\item $g$ is \textit{elliptic} if and only if $g$ has a unique fixed point in $\BH^2$;
\item $g$ is \textit{hyperbolic} if and only if $g$ has two fixed points on $\partial\BH^2$;
\item $g$ is \textit{parabolic} if and only if $g$ has a unique fixed point on $\partial\BH^2$,
\end{enumerate}
where $\partial \BH^2$ means the set of boundary points of $\BH^2$.
\end{definition}

Two M$\ddot{\text{o}}$bius transformations $f,g$ are \textit{conjugate} if there is a M$\ddot{\text{o}}$bius transformation $h$ such that $g=hfh^{-1}$. Any M$\ddot{\text{o}}$bius transformation is conjugate to one of the above three types. Indeed, a parabolic isometry is conjugate to $g(z)=z+s$, $s\in \R$. An elliptic one is conjugate to $g(z)=\frac{s}{-1/s z}=-s^2/z$, $s\in \R\setminus {0}$. A hyperbolic one is conjugate to $g(z)=\frac{sz}{1/s}=s^2 z$, $s\in \R\setminus {0}$. Moreover, the types of hyperbolic isometries are characterized by the trace.
\begin{theorem}\cite[Theorem 4.3.4]{Be} \label{classification}
Let $g(\neq \id)$ be any M$\ddot{\text{o}}$bius transformation in \eqref{Mobius transformations} with $\trace(g):=a+d$. Then 
\begin{enumerate}
\item $g$ is parabolic if and only if $\trace^2(g)=4$;
\item $g$ is elliptic if and only if $\trace^2(g)\in [0,4)$;
\item $g$ is hyperbolic if and only if $\trace^2(g)\in (4,+\infty)$.
\end{enumerate}
\end{theorem}



Let $f,g,h$ be elliptic isometries of order two with distinct fixed points $u,v,w$ respectively. We assume that $u,v$ and $w$ are not collinear. For the hyperbolic triangle with vertices $u,v$ and $w$, let $\alpha,\beta$ and $\gamma$ be the angles and let $a,b$ and $c$ be the lengths of the opposite sides respectively. It is known as the Sine rule for hyperbolic triangles that 
\begin{equation}
\frac{\sinh a}{\sin \alpha}=\frac{\sinh b}{\sin \beta}=\frac{\sinh c}{\sin \gamma}.
\end{equation}
Let $\lambda=\sinh a \sinh b \sin \gamma(=\sinh b \sinh c \sin \alpha=\sinh c \sinh a \sin \beta$). 
\begin{theorem}\cite[Theorem 11.5.1]{Be}\label{trace of any three elliptic}
The absolute value of the trace of any of the isometires
\begin{equation*}
fgh, hfg,ghf, hgf, fhg,gfh
\end{equation*}
is equal to $2\lambda$.
\end{theorem}

\begin{definition}\cite[\S 5.1. and \S 9.4.]{Be}\cite[\S 3.2.]{K}
$G\subset \text{M}\ddot{\text{o}}\text{b}^+(\BH^2)$ is \textit{elementary} if there is a finite $G$-orbit in $\BH^2\cup\partial \BH^2$.
A discrete subgroup of $\PSL(2,\R)$ is called \textit{Fuchsian group}. For any Fuchsian group $G$, the \textit{Dirichlet polygon} $D(p)$ with \textit{center} $p$ is the set
\begin{equation}
D(p)=\{z\in \BH^2\mid d(z,p)<d(z,g(p)) \text{ for all } g\in G\setminus \{\id\}\},
\end{equation}
where $d$ is the hyperbolic metric on $\BH^2$ and $p$ is any point in $\BH^2$ that is not fixed by any elliptic element of $G$.
\end{definition}

Let $G$ be a finitely generated non-elementary Fuchsian group. Any Dirichlet polygon $D(p)$ for $G$ is finite sided and topologically, $(D\cup \partial D)/G$ is a compact surface $S$ of some genus, say $g$, with a certain number of holes removed. It is independent of the choice of $p$, provided that $p$ is not an elliptic fixed-point of $G$. The genus $g$ does not depend on the choice of $D$. Let $r,s,t$ be the number of conjugacy classes of maximal elliptic cyclic subgroups with orders $m_1,\cdots, m_r$, the number of conjugacy classes of maximal parabolic cyclic subgroups, and the number of conjugacy classes of maximal hyperbolic boundary cyclic subgroups, respectively.



\begin{definition}\cite[Definition 10.4.1]{Be}\label{signature}
The symbol
\begin{equation*}
(g:m_1,\cdots,m_r;s;t)
\end{equation*}
\end{definition}
is called the \textit{signature} of $G$: each parameter is a non-negative integer and $m_j\geq 2$.

It is known that every Fuchsian group has a presentation of the following form (cf.\cite{HV},\cite{LS}).
\begin{equation}\label{group presentation}
\begin{array}{lcl}
\text{Generators}: & a_1,b_1,\cdots,a_g,b_g & \text{(Hyperbolic)}\\
&x_1,x_2,\cdots,x_r & \text{(Elliptic)}\\
&p_1,\cdots,p_s & \text{(Parabolic)}\\
&h_1,\cdots,h_t& \text{(Hyperbolic boundary elements)}\\
\end{array}
\end{equation}
\begin{equation}\label{relations}
\text{Relations}: x_1^{m_1}=x_2^{m_2}=\cdots=x_r^{m_r}=\prod_{i=1}^ga_ib_ia_i^{-1}b_i^{-1}\prod_{j=1}^rx_j\prod_{k=1}^sp_k\prod_{l=1}^th_l=1.
\end{equation}

If $s+t>0$, the Fuchsian group is a free product of cyclic groups (cf. \cite{HV},\cite[\S 3]{LS}):
\begin{equation}\label{free product of Fuchsian groups}
G\cong \Z_{m_1}\ast \cdots \ast \Z_{m_r}\ast F_v,
\end{equation}
where $F_v$ is a free group of rank $v=2g+s+t-1$. 

\begin{theorem}\cite[Theorem 11.5.2]{Be}\label{signature}
Let $f,g$ and $h$ be elliptic involutions which generate a non-elementary group $G$ and let $\lambda$ be given by Theorem \ref{trace of any three elliptic}.
\begin{enumerate}
\item If $\lambda>1$ then $G$ is discrete and has signature $(0:2,2,2;0;1)$.
\item If $\lambda=1$ then $G$ is discrete and has signature $(0:2,2,2;1;0)$.
\item If $\lambda<1$ then $G$ is discrete only if $\lambda$ is one of the values
\begin{equation*}
\cos(\pi/q), q\geq 3; \hspace{2mm} \cos(2\pi/q), q\geq 5; \hspace{2mm} \cos(3\pi/q), q\geq 7:
\end{equation*}
the possible signatures for $G$ are
\begin{equation*}
(0:2,2,2,q;0;0),(0:2,3,q;0;0),(0:2,4,q;0;0).
\end{equation*}
\end{enumerate}
\end{theorem}

For two elliptic isometries, we have the following theorem.
\begin{theorem}\cite[\S 5]{LU}
Let $A$ and $B$ be the elliptic generators of finite orders of $\mathfrak{A}$ and $\mathfrak{B}$, respectively. If $AB\neq \id$ and $AB$ has real fixed points, then the group $G$ generated by $A$ and $B$ is the free product of $\mathfrak{A}$ and $\mathfrak{B}$.
\end{theorem}

\begin{example}\label{two elliptic}
Let $f$ and $g$ are elliptic isometries of order $2$ without common fixed points. If $fg$ is hyperbolic, then the group $G=\langle f,g \rangle\cong \Z_2\ast \Z_2$.
\end{example}

\section{Proof}
\subsection{Proof of Theorem \ref{main theorem1}}
First of all, we will show that  $\Aut(X^{[2]})\cong\langle \iota_0,\iota_1,\iota_2\rangle$.

Recall that the natural automorphisms leave invariant the half-exceptional divisor $e$. Whereas for the non-natural automorphisms, we have the following lemma:
\begin{lemma}\label{non-natural}
Let $f\in \Aut(X^{[2]})$ be a non-natural automorphism of $X^{[2]}$. Then $f^*(e)=xh_1+\lambda e+yh_2$ with $xh_1+yh_2\in \Amp(X)$ and $\lambda<0$.
\end{lemma}
\begin{proof} 
Let $f^*(e)=(x,\lambda,y)$ with respect to the basis $\{h_1,e,h_2\}$ of $NS(X^{[2]})$. Since $f^*(2e)=f^*(E)\neq E$ and $f^*(E)$ and $E$ are both effective and irreducible, $(f^*(e),e)=-2\lambda\geq 0$. Since $(f^*(e))^2=(e)^2=-2$, $4(x^2+xy-y^2)-2\lambda^2=-2$, hence $\lambda\neq 0$, i.e., $\lambda<0$. Moreover, $(x,y)^2=4(x^2+xy-y^2)=2(\lambda^2-1)>0$, i.e., either $(x,y)\in \Amp(X)$ or $-(x,y)\in \Amp(X)$ by \eqref{ample cone of X}. Moreover, since $(1,0), (1,1)\in \Amp(X)$, $(1,-1,0)$ and $(1,-1,1)$ are ample on $X^{[2]}$ by Lemma \ref{ample divisors}, hence intersections with these imply that $x>0$, i.e., $(x,y)$ is an ample divisor of $X$.
\end{proof}

For a divisor $D\in NS(X^{[2]})$, let $[D]_e$ be the $e$-coordinate with respect to the basis $\{h_1,e,h_2\}$. 

\begin{theorem}\label{main theorem1-1}
$\Aut(X^{[2]})\cong\langle \iota_0,\iota_1,\iota_2 \rangle.$
\end{theorem}

\begin{proof} Let $f\in \Aut(X^{[2]})$ be a non-natural automorphism of $X^{[2]}$. By Lemma \ref{non-natural}, we have that $f^*(e)=(x,\lambda,y)$ with $(x,y)\in \Amp(X)$ and $\lambda<0$. If $[\iota_k^*f^*(e)]_e\geq 0$ for some $k\in \Z$, then again by Lemma \ref{non-natural}, $\iota_k^*f^*(e)=e$, i.e., $\iota_k^*f^*=g^{n*}$ for some $n\in \Z$. Now by \eqref{injection}, $g^n=f\iota_k$ in $\Aut(X^{[2]})$. Since $g=\iota_0\iota_1\iota_2$ and by Theorem \ref{relations}, $f\in \langle \iota_0,\iota_1,\iota_2 \rangle.$
Next, suppose that for all $k\in \Z$, $[\iota_k^*f^*(e)]_e<0$. Then we claim the following.
\begin{claim*}\label{claim}
For some $m\in \Z$, $[\iota_m^*f^*(e)]_e>\lambda$.
\end{claim*}

\begin{claimproof*}
Suppose that for all $k\in \Z$, $[\iota_k^*f^*(e)]_e\leq \lambda$.
In \text{\cite[Example 1]{L}}, $\iota_0^*,\iota_1^*,\iota_2^*\in O(H^2(X^{[2]},\Z))$ are given as follows with respect to the basis $\{h_1,e,h_2\}$:
\begin{equation}\label{iotas}
\iota_0^*=\begin{bmatrix} 3&2&2\\-4&-3&-2\\0&0&-1 \end{bmatrix}, \hspace{2mm} \iota_1^*=\begin{bmatrix} 5&2&-2\\-6&-3&2\\6&2&-3 \end{bmatrix},\hspace{2mm}  \iota_2^*=\begin{bmatrix} 27&4&-16\\-14&-3&8\\42&6&-25 \end{bmatrix}.
\end{equation}
In particular, $[\iota_0^*f^*(e)]_e=-4x-3\lambda-2y\leq \lambda$, i.e., $2x+y\geq -2\lambda$. Similarly, $[\iota_1^*f^*(e)]_e=-6x-3\lambda+2y\leq \lambda$, i.e., $3x-y\geq -2\lambda$ and $[\iota_2^*f^*(e)]_e=-14x-3\lambda+8y\leq \lambda$, i.e., $7x-4y\geq -2\lambda$. Moreover, by Theorem \ref{relations}, $\iota_{3l+k}^*=g^{-l*}\circ\iota_k^*\circ g^{l*}$ for $l\in \Z$ and $k=0,1,2$. Since $g$ is natural, $g^*(e)=e$, hence $[\iota_{3l+k}^*f^*(e)]_e=[\iota_k^*g^{l*}f^*(e)]_e$, where $k=0,1,$ or $2$. Let $g^{l*} f^*(e)=(x',\lambda,y')$ with $(x',y')\in \Amp(X)$. If we let $r'=y'/x'$, then  $(g^{l*} f^*(e))^2=(e)^2=-2$ gives
\begin{equation}\label{same intersection number}
4x'^2(1+r'-r'^2)=2\lambda^2-2.
\end{equation}

Now for $k=0$, $[\iota_0^*g^{l*}f^*(e)]_e=-4x'-3\lambda-2y'\leq \lambda$ gives that $2x'+y'\geq -2\lambda$. By taking the squares, we have that $(x'(2+2r'))^2\geq 4\lambda^2$. By substituting \eqref{same intersection number}, we have that $x'^2(4+4r'+r'^2)\geq 8x'^2(1+r'-r'^2)+4$ or $x'^2(9r'^2-4r'-4)\geq 4$. However, the left hand side is negative if  $\frac{2-2\sqrt{10}}{9}<r'<\frac{2+2\sqrt{10}}{9}$.  

Similarly for $k=1$, we have  $x'^2(9r'^2-14r'+1)\geq 4$, and the left hand side is negative if  $\frac{7-2\sqrt{10}}{9}<r'<\frac{7+2\sqrt{10}}{9}$. For $k=2$, we have  $x'^2(24r'^2-64r'+41)\geq 4$, and the left hand side is negative if  $\frac{16-\sqrt{10}}{12}<r'<\frac{16+\sqrt{10}}{12}$. Since the ample cone of $X$ is given in \eqref{ample cone of X}, 
for some power $l$, at least one of three inequalities induces a contradiction, hence $[\iota_m^{*}f^*(e)]_e>\lambda$. 
\end{claimproof*}
If we continue this process at least $\vert \lambda \vert-1$ times, we get a contradiction to $[\iota_k^*f^*(e)]_e<0$  for all $k\in \Z$. Hence, we may assume that for some $\iota^*=\iota_{k_1}^*\iota_{k_2}^*\cdots\iota_{k_j}^*$, $[\iota^*f^*(e)]_e\geq 0$. Then again by Lemma \ref{non-natural}, $\iota^*f^*(e)=e$, i.e., $f\iota=g^l$ in $\Aut(X^{[2]})$ for some $l\in \Z$. Again by $g=\iota_0\iota_1\iota_2$ and by Theorem \ref{relations}, $f\in \langle \iota_0,\iota_1,\iota_2 \rangle.$
\end{proof}

\vspace{3mm}

\subsection{}
Next, we will show that $\langle \iota_0,\iota_1,\iota_2\rangle\cong \Z_2\ast \Z_2\ast\Z_2$.

Note that the Picard lattice $NS(X^{[2]})$ of Hilbert square of the K3 surface $X$ has signature $(1,2)$. For an ample divisor $H$ on $X^{[2]}$, we consider the \textit{light cone} of $NS(X^{[2]})$: 
\begin{equation}
\mathcal{L}^+:=\{D\in \NS(X^{[2]})\otimes \R \mid D^2>0, D.H>0\}.
\end{equation}
If we let $H=(1,-1,0)$ with respect to the basis $\{h_1,e,h_2\}$, then  
\begin{equation}\label{positive cone}
\mathcal{L}^+=\{(x,\lambda,y) \in \NS(X^{[2]})\otimes \R \mid 2(x^2+xy-y^2)>\lambda^2, 2x+y+\lambda>0\}.
\end{equation}
Since the signature of $\NS(X^{[2]})$ is $(1,2)$, $\mathcal{L}^+/\R^+$ with distance $\vert AB\vert$ defined by 
\begin{equation}\label{distance}
\Vert A\Vert \Vert B\Vert \cosh\vert AB\vert=(A.B)
\end{equation}
is isomorphic to $2$-dimensional hyperbolic geometry $\mathbb{H}^2$, where $\Vert A\Vert=\sqrt{(A.A)}$. 

Note that for each $k\in \Z$, $\iota_k^*$ preserves distances of $\sL^+/\R^+$, hence $\iota_k^*\in \Aut(\BH^2)$ and it is an involution. Fixed points of $\iota_0^*$, $\iota_1^*$ and $\iota_2^*$ are $[1,-1,0]$, $[1,-1,1]$, and $[2,-1,3]$, respectively. Here $[x,\lambda,y]$ denotes a class of points modulo $\R^+$. Note that since these fixed points correspond to interior points of $\BH^2$, $\iota_0^*,\iota_1^*,\iota_2^*$ are all elliptic isometries.

\begin{remark}
Note that fixed points $[1,-1,0],[1,-1,1],[2,-1,3]$ correspond to ample divisors of square $2$ on $X^{[2]}$ with respect to the Beauville-Bogomolov-Fujiki's form.
\end{remark}

In the rest, we will show that there is no relation between $\iota_0^*,\iota_1^*,\iota_2^*$. 

It is known that if $\gamma_1,\gamma_2$ are two geodesics in $\BH^2$ and $z_i\in \gamma_i$ are two points on them, then there exists an isometry $g\in M\ddot{\text{o}}b^+(\BH^2)$ such that $g(\gamma_1)=\gamma_2$ and $g(z_1)=z_2$. Hence we may assume that two points of $[1,-1,0],[1,-1,1],[2,-1,3]$ map to $(0,1),(0,k)\in \BH^2$ with $k>1$ on the $y$-axis of $\BH^2$, a hyperbolic line. Then the remaining point maps to $(b,m)\in \BH^2$ with $b\neq 0$. Moreover, $\iota_0^*,\iota_1^*,\iota_2^*$ correspond to the elliptic isometries $T,U,V$ whose fixed points are $(b,m), (0,1), (0,k)$ respectively. Hence these M$\ddot{\text{o}}$bius transformations correspond to the following matrices in $\SL(2,\R)$; 
\begin{equation}\label{U and V}
U=\begin{bmatrix} 0& 1\\ -1&0 \end{bmatrix}, \hspace{2mm}V=\begin{bmatrix} 0 &k\\ -1/k&0 \end{bmatrix}
\end{equation}
and 
\begin{equation}\label{T}
T=\begin{bmatrix} 1&b \\ 0&1 \end{bmatrix} \begin{bmatrix} 0 &m\\ -1/m& 0\end{bmatrix} \begin{bmatrix} 1&-b \\ 0&1 \end{bmatrix}=\begin{bmatrix} b/m &-(b^2+m^2)/m\\ 1/m& -b/m\end{bmatrix}. 
\end{equation}
Note that all traces of $T,U,V$ are zero, i.e, elliptic isometries with fixed points $(b,m),(0,1),(0,k)$ respectively. Note also that the first matrix in \eqref{T}
corresponds to the translation $z\mapsto z+b$.

Consider the group $G$ generated by $T,U,V$. Note that any two of three generates a free group $\Z_2\ast \Z_2$ (Example \ref{two elliptic}). Hence any element in $G$ can be written as $TX_1TX_2\cdots TX_k$, where each $X_i$ is of the form
\begin{equation}\label{X forms}
(UV)^n,(UV)^nU,(VU)^s,(VU)^sV
\end{equation}
in the free group $\Z_2\ast \Z_2$ generated by $U,V$, where $n,s\in \Z$.
\begin{lemma}\label{TX hyperbolic}
$TX_i$ is hyperbolic for any $X_i$ in \eqref{X forms}.
\end{lemma}
  
\begin{proof}
We use the classification of a hyperbolic isometry by the trace (Theorem \ref{classification}). First, 
\begin{equation}
TU=\begin{bmatrix} (b^2+m^2)/m& b/m \\ b/m& 1/m \end{bmatrix}
\end{equation}
and its trace is $b^2/m+m+1/m>2$, hence it is hyperbolic. Similarly,
\begin{equation}
T(UV)^sU=TU(VU)^s=\begin{bmatrix} \frac{b^2+m^2}{m}k^s&* \\ *&\frac{1}{m}k^{-s} \end{bmatrix}
\end{equation}
and its trace is $\frac{b^2}{m}k^{s}+mk^s+(mk^s)^{-1}>2$, hence it is hyperbolic.
Moreover,
\begin{equation}
T(VU)^{s-1}V=TU(UV)^s=\begin{bmatrix} \frac{b^2+m^2}{m}k^{-s}&* \\ *&\frac{1}{m}k^{s} \end{bmatrix}
\end{equation}
and its trace is $\frac{b^2}{m}k^{-s}+mk^{-s}+(mk^{-s})^{-1}>2$, hence it is also hyperbolic.

Next,
\begin{equation}
T(VU)^s=TU(UV)^sU=\begin{bmatrix} -\frac{b}{m}k^s&* \\ *&\frac{b}{m}k^{-s} \end{bmatrix}
\end{equation}
and its trace is $\vert \frac{b}{m}(k^{s}-k^{-s})\vert>2$ by induction. Indeed for $s=1$, $TVU$ corresponds to one of $\iota_i^*\iota_j^*\iota_k^*$ where $i,j,k \in \{0,1,2\}$ are distinct. It is enough to show that $\iota_2^*\iota_1^*\iota_0^*$ is hyperbolic by Theorem \ref{trace of any three elliptic}.
\begin{equation}
\iota_2^*\iota_1^*\iota_0^*=\begin{bmatrix} 5& 0&8\\0&1&0\\8&0&13 \end{bmatrix}
\end{equation}
has two fixed points 
\begin{equation}
((-1+\sqrt{5})/2,0,1)\hspace{2mm}  \text{ and }\hspace{2mm}(-(1+\sqrt{5})/2,0,1)
\end{equation}
on the boundary of $\sL^+$, hence they correspond to two points on $\partial \BH^2$.
Thus $TVU$ is hyperbolic, i.e., the absolute value of its trace $\vert \frac{b}{m}(k-k^{-1})\vert >2$. For $s=2$, 
\begin{equation}
\begin{array}{lcl}
\vert \frac{b}{m}(k^2-k^{-2})\vert &=&\vert\frac{b}{m}(k-k^{-1})\vert(k+k^{-1})\\
&>&\vert \frac{b}{m}(k-k^{-1})\vert\\
&>&2.\\
\end{array}
\end{equation}
Now suppose that $\vert k^n-k^{-n}\vert \geq \vert k^{n-1}-k^{-n+1}\vert$. Then 
\begin{equation}
\begin{array}{lcl}
\vert k^{n+1}-k^{-n-1}\vert &=&\vert (k^n-k^{-n})(k+k^{-1})-(k^{n-1}-k^{-n+1})\vert \\
&\geq &  \vert (k^n-k^{-n})(k+k^{-1}-1)\vert \\
&\geq & \vert k^n-k^{-n}\vert. \\
\end{array}
\end{equation}
Hence the trace  $\vert \frac{b}{m}(k^{s}-k^{-s})\vert>2$ and $T(VU)^s$ is hyperbolic.

Similarly,
\begin{equation}
T(UV)^{s+1}=TU(VU)^sV=\begin{bmatrix} -\frac{b}{m}k^{-s-1}&* \\ *&\frac{b}{m}k^{s+1} \end{bmatrix}
\end{equation}
and its trace is $\vert \frac{b}{m}(k^{s+1}-k^{-s-1})\vert>2$, hence $T(UV)^{s+1}$ is hyperbolic.
\end{proof}


Now by Theorem \ref{signature}, $G$ has the signature $(0;2,2,2;0,1)$ with $g=0,m_1=m_2=m_3=2,s=0,$ and $t=1$. Since $s+t>0$,  $G$ is a free product of three elliptic elements by \eqref{free product of Fuchsian groups},  i.e., $G=\langle T,U,V\rangle\cong\langle \iota_0^*,\iota_1^*,\iota_2^*\rangle\cong \Z_2\ast \Z_2\ast \Z_2$. Now by \eqref{injection}, in $\Aut(X^{[2]})$, $\iota_0,\iota_1,\iota_2$ has no relation, hence $\Aut(X^{[2]})\cong \langle \iota_0,\iota_1,\iota_2\rangle\cong \Z_2\ast \Z_2\ast \Z_2$.





\subsection*{Acknowledgements}
The author would like to thank Professor Keiji Oguiso for letting me to read his preprint.

\end{document}